\documentclass[11pt]{amsart}
\usepackage{amsmath,amsthm,amscd,amssymb,graphics,wasysym,float, graphicx}
\usepackage[all]{xy} 

\newtheorem{thm}{Theorem}[section]

\newtheorem{cor}[thm]{Corollary}
\newtheorem{prop}[thm]{Proposition}

\newtheorem{lma}[thm]{Lemma}
\newtheorem{rem}[thm]{Remark}

\theoremstyle{definition}
\newtheorem{dfn}[thm]{Definition}

\theoremstyle{remark}
\newtheorem{remark}[thm]{Remark}

\numberwithin{equation}{section}

\newcommand{\bbr}{\begin{remark}}        
\newcommand{\eer}{\end{remark}}

\newcommand{\bea}{\begin{eqnarray}}
\newcommand{\eea}{\end{eqnarray}}
\newcommand{\bmini}{\begin{center}\begin{minipage}{5in}}
\newcommand{\emini}{\end{minipage}\end{center}}
\newcommand{\R}{\mathbb{R}}

\newcommand{\Z}{\mathbb{Z}}

\newcommand{\Strata}{{\mathcal{ST}}}

\newcommand{\pa}{\partial}

\begin{document}

\title{Transverse string topology  and the cord algebra}
\author{Somnath Basu \and Jason McGibbon \and Dennis Sullivan \and Michael Sullivan}

\address{Somnath Basu, Department of Mathematical Sciences \\ Binghamton University}
\email{somnath@math.binghamton.edu}

\address{Jason McGibbon, Department of Mathematics\\ Massachusetts Institute of Technology}
\email{mcgibbon@math.mit.edu}

\address{Dennis Sullivan, Mathematics Department\\
State University of New York, Stony Brook}
\email{dennis@math.sunysb.edu}

\address{Michael Sullivan, Department of Mathematics and Statistics\\ University of Massachusetts}
\email{sullivan@math.umass.edu}



\begin{abstract}
We define a coalgebra structure for open strings transverse to any framed codimension 2 submanifold $K\subset M$. When the submanifold is a knot in $\R^3,$ we show this structure recovers a specialization of Ng cord algebra \cite{Ng3}, a non-trivial knot invariant which is not determined by a number of other knot invariants.
\end{abstract}

\thanks{JM was partially supported by NSF grant DMS-0943787.
	DS was partially supported by NSF grant DMS-0757245.
	MS was partially supported by NSF grant DMS-1007260.}

\maketitle


\section{Introduction}
\label{Introduction.sec}

There are many geometrically defined operations in open-closed string topology. 
See \cite{DPS}, for example. Some of them pass to homology, while others have boundary terms (or ``anomalies") and do not. In \cite{CS} and \cite{Po}, the splitting of closed strings has an anomaly which is canceled by working modulo constant strings. We do not apply this to the splitting of open strings because the unit of the combining operation is made of constant
strings. Still, we desire to remove this anomaly for several reasons which will be explained in this note and others \cite{Ba, McG, SS}.


We treat the anomaly for splitting, which occurs in ``traditional" (universal) string topology, by passing to the subset of open strings which are transverse to the submanifold.
We call the resulting theory {\bf{transverse string topology}} to distinguish it from the well-studied universal string topology. This is a drastic change in perspective because the algebraic topology of the space of strings changes in a subtle way. The value of doing this was motivated by the desire to relate this structure on open strings for classical knots to the Ng cord algebra defined in \cite{Ng2,Ng3}. As noted independently by different people, one of the defining relations of the Ng cord algebra resembles a relation related to splitting open strings.

Indeed, splitting transverse open strings defines a differential coalgebra structure where the differential has two pieces: an internal boundary operator (without anomaly); and a term which resolves the concatenation of two open strings which do not intersect the submanifold in their interiors. This term can be recognized via the bar construction for a partially-defined combining operation on such open strings. Our main result in this note, which we restate more precisely as Theorem \ref{MainTheorem.thm}, is the following:

\begin{thm}\label{thmintro}
The zeroth homology of the cobar construction on the coalgebra of vertically transverse open strings, relative to a cross section of the normal bundle, determines a ring isomorphic to a specialization of Ng's cord algebra. The isomorphism is as algebras over the group ring of an infinite cyclic group.
\end{thm}

In \cite{Ng1} Ng shows that the cord algebra is a powerful knot invariant. Combining the isomorphism of Theorem \ref{thmintro} with his results we find the following:

\begin{cor} 
\label{NgComputation.cor}
The operations on (transverse) open strings defines a non-trivial
knot invariant not determined by any of the following: Alexander polynomial, Jones polynomial, HOMFLY polynomial, Kauffman polynomial, signature, Khovanov invariant, and Ozsvath-Szabo invariant.
\end{cor}

\begin{proof}
This follows from our main result, \cite[Proposition 8.4]{Ng1} and \cite[Proposition 4.3]{Ng3}. 
\end{proof}

One way to understand what is really going on here is based on \cite{Ba}.
That thesis defines the partial multiplication alluded to above  from a categorical perspective where the (cobar construction) of (bar construction) is an equivalence. This allows one to interpret the cobar of the open string coalgebra, in terms of an algebraic construction, as a twisted Pontrjagin ring of the based loop space of the complement. In the case of knots (in $S^3$) this can be realized as a version of the Ng cord algebra. 
As an application of the theory developed in the thesis, the transverse string topology in \cite{Ba} distinguishes two homotopic but non-homeomorphic Lens spaces.

This explanation itself was motivated in two ways by \cite{McG}. First, that thesis made use of a cobar of the coalgebra for a specific set of open strings. Secondly, that construction required transversality, which was already known
to remove the internal anomaly and to make the partial algebra structure possible.
The thesis goes on to recover the cord algebra, as well as higher degrees of relative contact homology (for $\Z$-coefficients).

The cord algebra is isomorphic to a particular version of relative contact homology defined using pseudo-holomorphic disks \cite{EENS}.
When there is no submanifold, a morphism between a similar contact homology and a closed string topology theory
was introduced in \cite{CL}. This morphism is natural in that it comes from geometry.  A relative version of this morphism has been proposed in \cite{CELN}.

We would like to emphasize that we do not assume that ``the information of the complement" is directly given to us. Rather, we derive what we need from the transverse open strings.
With this point of view but from a different perspective, \cite{SS} via 
a cyclic Hochschild construction proves that transverse string topology of the knot
recovers the bracket (and other information) on the equivariant loop space of the complement. Via \cite{CG}, this information is enough to recover the Seifert and hyperbolic graph structure of the JSJ-decomposition \cite{JSJ} of the complement. Hence transverse string topology again provides non-trivial knot invariants.

The paper is organized as follows.
In Section \ref{TransverseStringTopology.sec} we review our version of transverse string topology and define the string algebra $A_{St}$ via the homology of its cobar.
In Section \ref{Ng.sec} we review 
one of Ng's presentations of his cord algebra $A_{Ng}.$ We then prove the main isomorphism Theorem \ref{MainTheorem.thm}
 relating the two algebras.

\section{Transverse string topology}
\label{TransverseStringTopology.sec}

In this section, we review a particular space of open strings introduced in \cite{SS} and some natural algebraic structures on it.
We also discuss the cobar construction and compute its zeroth homology, which leads to our definition of the string algebra $A_{St}.$
We compute $A_{St}$ in the case of the unknot in $\R^3.$

\subsection{The space of strings}
\label{Space.sec}

Let $K \subset M$ be any framed codimension 2 submanifold.
Let $N$ be a tubular neighborhood of $K.$
Use the framing to trivialize the normal bundle $N \rightarrow K$ and to define antipodal sections 
\[
\Gamma^+, \Gamma^-: K \rightarrow \pa N. 
\]
For points $P,Q$ in the fiber $N_q$ over the point $q \in K,$ 
let $Q-P$ denote the oriented line segment in $N_q$ 
from $P$ to $Q.$

\begin{dfn}
\label{OpenSpace.def}
Let $\Strata = \Strata(K, \Gamma^\pm)$ be the set of open strings (smooth maps
of the unit interval into $M$) starting and ending at $K$ transverse
to $K,$ and which intersect $N$ in a radial way at $\Gamma^+$ and $\Gamma^-$ only. More specifically, if a string $\omega \in \Strata$ starts (resp. ends) at the point $q_0$ (resp. $q_1$) in $K$ then $\omega$ first leaves (resp. last enters) $N$
as the line segment $\Gamma^+(q_0)-q_0$ (resp. $q_1 - \Gamma^-(q_1) $). 
Near any other intersection point $q$ of $\omega$ and $K,$ 
$\omega \cap N =  \Gamma^+(q)- \Gamma^-(q)$ for $q \in K.$ 
\end{dfn}
\vspace*{-0.2cm}
\begin{figure}[h]
\begin{center}
\includegraphics[scale=0.55]{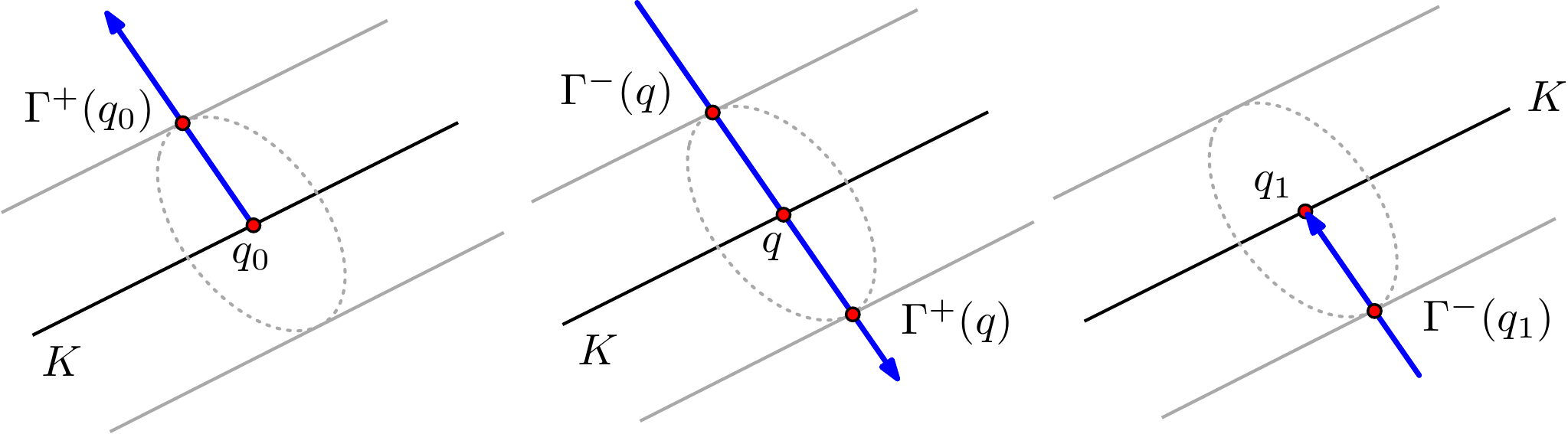}\vspace{0.2cm}
\caption{The various intersection points of an open string.}
\end{center}
\end{figure}


With any reasonable topology,  we can write the strings as a disjoint union
\[
\Strata  = \bigsqcup_{n \ge 0} \Strata_n
\]
where $\Strata_n$ is the stratum of strings in $\Strata$ intersecting  $K$ exactly $n$ times (excluding the start
and endpoint).

\begin{remark} 
We could also consider the space of string that are transversal to the submanifold and may intersect with any tangent condition (dropping the $\Gamma^{\pm}$ condition). For chains on the space, the subsequent operations we introduce are only ``partially defined" subject to tangent matching conditions. Nevertheless, the ideas of this paper still hold using the cobar construction for partially-defined structures (bicomodules) as presented in \cite{Ba}.
\end{remark}

\subsection{Algebraic structures}
\label{AlgebraicStructures.sec}



 
\begin{dfn}
\label{Resolve_i.dfn}
Let $\omega \in \Strata_n$ be a string with its $i^{\textup{th}}$ internal intersection at $q \in K.$ 
Consider the unit circle in $N_q,$ oriented by the co-orientation of $N,$ starting
and ending at $\Gamma^-(q).$ 
Let $e(q)$ denote the half of this circle from $\Gamma^-(q)$
 to $\Gamma^+(q)$ whose orientation disagrees with the circle,
 and $m(q)$ denote the half from $\Gamma^-(q)$
 to $\Gamma^+(q)$ whose orientation agrees.
Define the $i^\textup{th}$ {\bf{resolve}} of $\omega,$ $R_i(\omega)$, to be sum of a smooth approximation of two copies of the original curve $\omega:$ one with the line segment $\Gamma^+(q) - \Gamma^-(q)$ replaced by $e(q)$ and the other with the line segment replaced by $m(q).$
\end{dfn}
\begin{figure}[h]
\begin{center}
\includegraphics[scale=0.55]{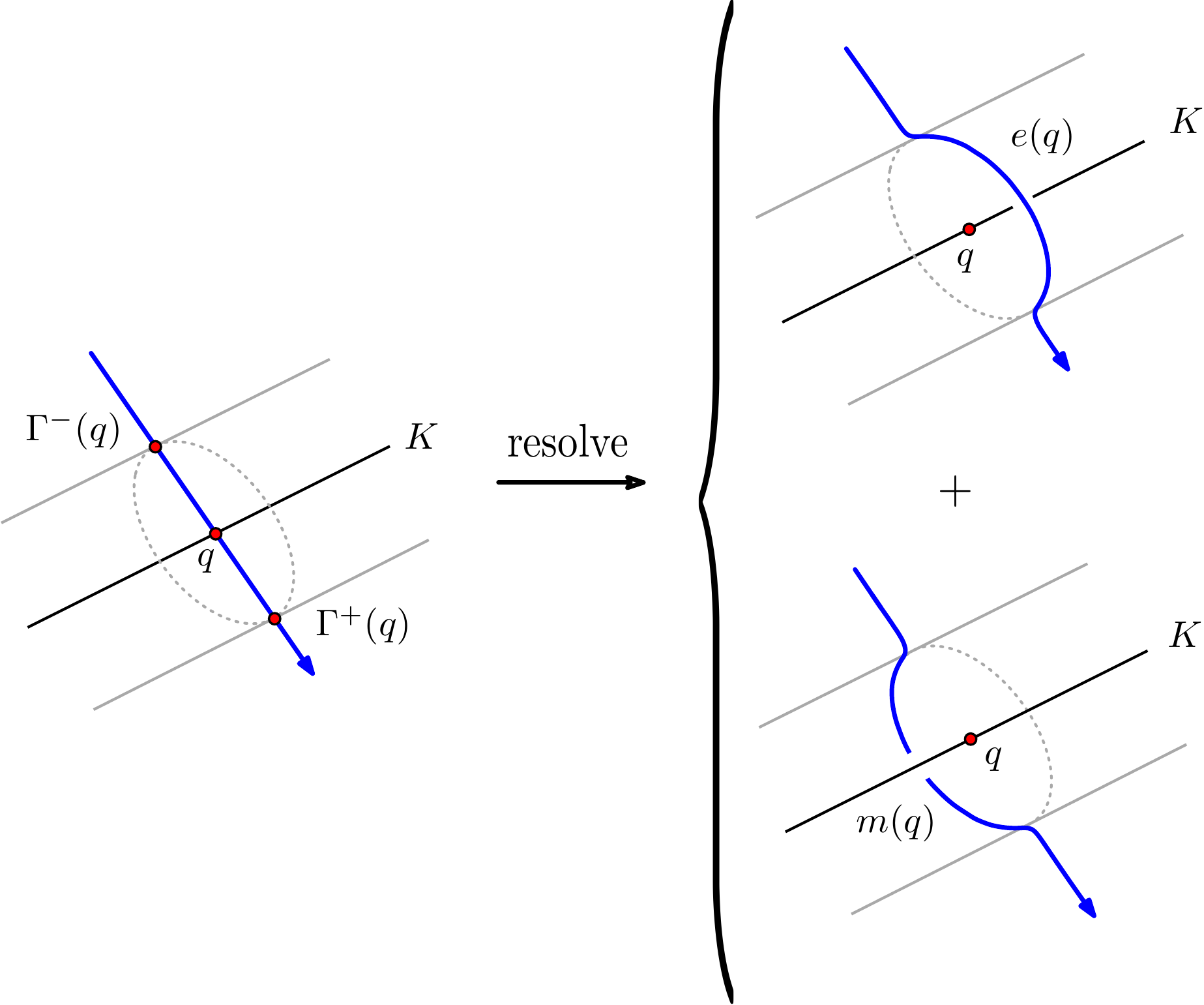}\vspace{0.2cm}
\caption{The resolve operation at a point.}
\end{center}\vspace*{-0.3cm}
\end{figure}

\begin{dfn}
\label{Split_i.dfn}
Let $\omega \in \Strata_n,$  with $1 \le i \le n.$
 Define the $i$-th {\bf{split}} of $\omega,$ $\Delta_i(\omega),$
 to be the ordered pair of curves $\omega' \in \Strata_{i-1}$ and $\omega''
 \in \Strata_{n-i-1}$ obtained by splitting $\omega$ at its $i$th intersection
 and reparameterizing the split interval into two unit interval domains. More formally, given $\omega:[0,1]\to M$ such that $w(\tau)=q\in K$ is an intersection point, we define the open strings
\begin{displaymath}
\omega':[0,1]\to M, t\mapsto \omega(\tau t)
\end{displaymath}
\begin{displaymath}
\omega'':[0,1]\to M, t\mapsto\omega\left(\tau + (1-\tau)t\right) .
\end{displaymath}
\end{dfn} 



Let $C_*(\Strata)$ be the singular chain complex with topological
boundary map $\pa.$
We set the grading so that an $n$-cell in $\Strata_m$ has grading
$n+m+1.$ 

\begin{dfn}
\label{Resolve.def}
Define the $i$-th {{resolve map}}, $R_i: C_k(\Strata) \rightarrow C_{k-1}(\Strata)$ by the pointwise defined map
$R_i(\omega)$ from Definition \ref{Resolve_i.dfn}.
Define the degree $-1$ {\bf{resolve map}}, $R: C_k(\Strata) \rightarrow C_{k-1}(\Strata)$
for each cell $c \in C_{k-l}(\Strata_{l-1})$ as the alternating sum 
\begin{equation}
\label{R.eq}
R(c) = \sum_{i=1}^{l-1} (-1)^{i-1}R_i(c)
\end{equation}
and extend to all of $C_k(\Strata)$ by linearity.
\end{dfn}
\begin{figure}[h]
\begin{center}
\includegraphics[scale=0.50]{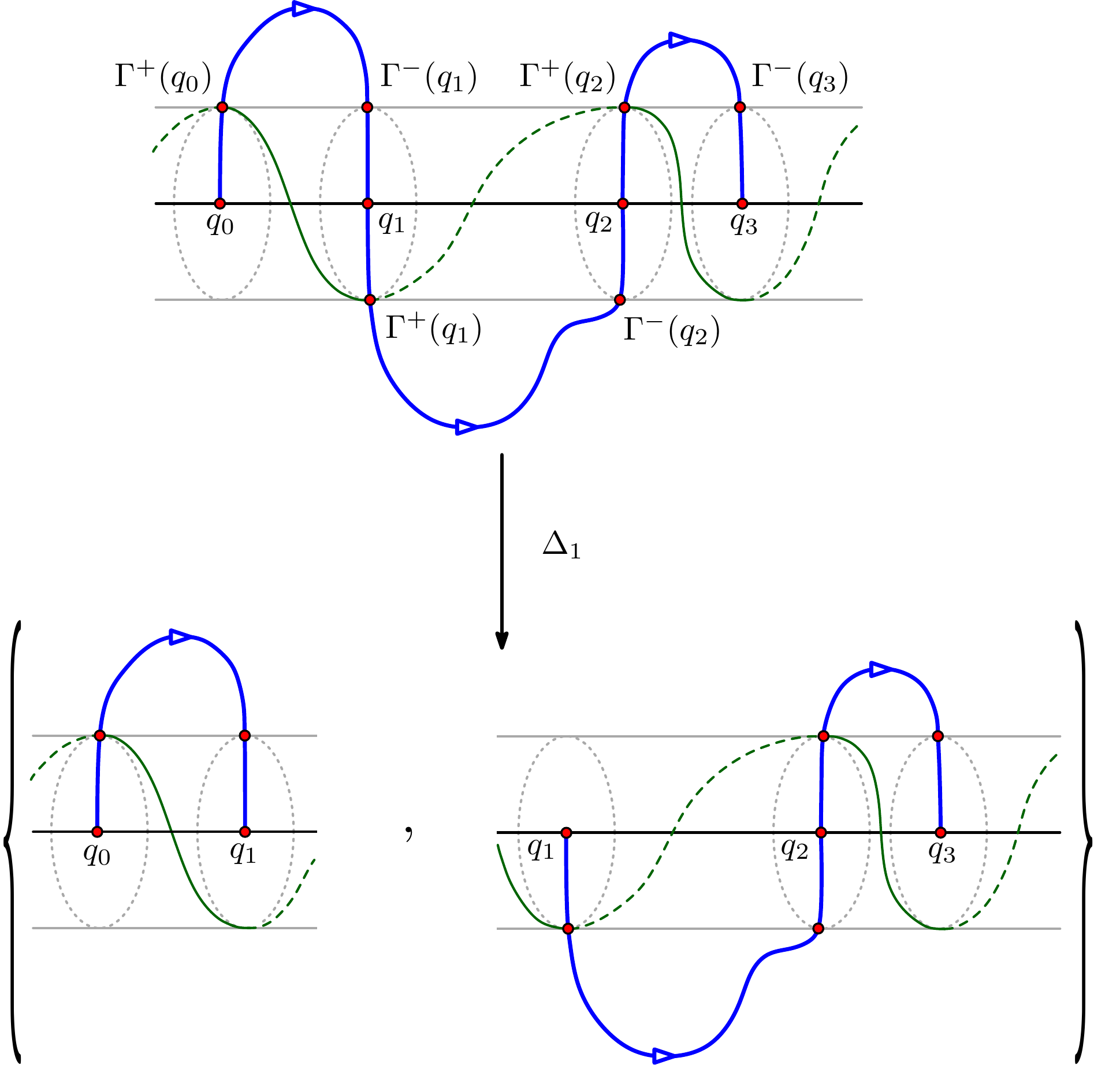}\vspace{0.2cm}
\caption{A picture of $\Delta_1$ in the coproduct.}
\end{center}\vspace*{-0.3cm}
\end{figure}

\begin{dfn}
Similarly, define the degree 0 {\bf{coproduct map}}
\[
\Delta:  C_k(\Strata) \rightarrow   \bigoplus_{i+j = k} C_i(\Strata) \otimes C_j(\Strata), \quad
\Delta(x) = \sum_{i=1}^{n} (-1)^{i-1} AW \circ \Delta_i(x)
\]
by composing the Alexander-Whitney map $AW$ with the pointwise-defined maps
\[
\omega \mapsto \Delta_i(\omega).
\]
\end{dfn}

A straightforward check reveals

\begin{lma}
\label{WellDefined.lma}
$R$ and $\Delta$ commute with $\pa,$ each other, and square to zero.
\end{lma}

\begin{remark}
\label{Dependence.rmk}
The maps $R$ and $\Delta$  at the chain level depend on several choices such
as  the reparameterization of the string's unit interval domain
and the smooth approximation in Definition \ref{Resolve_i.dfn}.
However, the induced maps on homology are independent of such
choices.
\end{remark}


\begin{cor}
\label{dgca.cor}
$(C_*(\Strata), \pa + R, \Delta)$ is a differential graded co-algebra (dgca). The homotopy type of the dgca is independent of the choices made in its construction.
\end{cor}


\subsection{Homology of the cobar}
\label{Cobar.sec}

Label the dgca from Corollary \ref{dgca.cor}
\[
\mathcal{C} := (C_*(\Strata), \pa + R, \Delta).
\]

Let $T^{red}(\cdot)$  denote the non-unital, i.e. reduced, tensor algebra.
Note that everything in $\mathcal{C}$ has positive degree and thus the dgca has no counit.
In this case, the cobar  $\Omega \mathcal{C}$  of $\mathcal{C}$
is the differential graded algebra (dga) 
\[
\Omega \mathcal{C} = (T^{red}(C_*(\Strata))[-1], \pa + R + \Delta, \otimes)
\]
where
$[-1]$ denotes performing a grading shift of $-1,$ $\otimes$ is the tensor product structure, and $\pa + R+\Delta$ is extended via derivation over $\otimes$ as a differential.
Note that from our grading conventions for $C_*(\Strata),$
$\Omega \mathcal{C}$ has only non-negative grading.

\begin{dfn}
Given a ring $R$ with product $*$ and an element $\alpha \in R,$ let the {\bf ring} $\tilde{R}$ {\bf twisted by} $\alpha$ be the set $R$ with product $\tilde{*}$ defined by $x \tilde{*} y := x * \alpha *y.$
\end{dfn}

Fix a point $q$ in the boundary of the tubular neighborhood of the submanifold $K.$
Let 
\[
[m], [e] \in   \pi_1(M \setminus K) := \pi_1(M \setminus K;q)
\]
 denote the homotopy classes of the meridional loop $m$ and the constant loop $e,$ respectively.
The cross-section from the framing of the normal bundle $N$ of $K$ defines a map $\pi_1(K) \rightarrow \pi_1(M \setminus K).$
Let $L$ be the image of this map.

\begin{prop}
\label{H_0.prop}
There is a ring isomorphism  from 
the degree 0 part of the homology of the dga 
$\Omega \mathcal{C}$  to the group ring $\widetilde{\Z \pi_1}(M \setminus K)$ twisted by the element 
$[m] +[e],$ modulo the two-sided ideal generated by 
\[
\{[lx]-[x], [xl] - [x]\,\, | \,\, [x] \in \pi_1(M \setminus K), [l] \in L\}.
\] 
Here $[xl]$ is the homotopy class of the concatenation of the loops $x$ and $l,$ based at $q.$
\end{prop}

\begin{proof}


The 0-chains of $\Omega \mathcal{C}$ are tensor products of arbitrary length of zero chains of $\Strata_0.$
Any 1-boundary is a linear combination of three types of 1-chains of $\Omega \mathcal{C}:$ 
\begin{enumerate}
\item
tensor products of the above zero-chains and one factor (in the tensor product) which is a smooth 1-family in $\Strata_0$ where the start and endpoints of the strings remain in a contractible set of $K;$ 
\item
tensor products of the above zero-chains and one factor which is a string with one internal intersection; and,
\item
tensor products of the above zero-chains and one factor which is a smooth 1-family in $\Omega_0$ where either the start or endpoints of the strings traces out a representative of 
$[l] \in L.$
\end{enumerate}


From the first type of 1-boundary, it is clear that we can choose a set of zero-chains (cycles)
which generate $H_0(\Omega \mathcal{C})$ and which are represented by strings in $St_0$ starting and ending at $q.$ 
Moreover, if we mod out these zero-chains by 1-boundaries of the first type we get an isomorphism from the algebra $A_1$ generated by these equivalence classes to the reduced tensor algebra of the vector space generated by $\pi_1(M \setminus K).$



Let $A_2$ be the algebra generated by equivalence classes of elements of $A_1$ where
two generators of $A_1$ are equivalent if they differ by a 1-boundary of the second type. The map from the previous paragraph then defines an algebra isomorphism from $A_2$ to the group ring $\widetilde{\Z\pi_1}(M \setminus K),$ twisted by the element $[e]+[m].$

Note that $H_0(\Omega \mathcal{C})$ is $A_2$ mod the image of the 1-boundaries of the third type; therefore $H_0(\Omega \mathcal{C})$ is isomorphic to  $\widetilde{\Z \pi_1}(M \setminus K)$ after dividing by the ideal in the proposition's statement.
\end{proof}

\subsection{A string algebra over the group ring of an infinite cyclic group}
\label{StringAlgebra.sec}

Let $\mathfrak{e}$ and $\bar{\mathfrak{m}}$ denote the elements in 
$H_0(\Omega \mathcal{C})$ which correspond to the equivalence classes of $[e]$ and $[\bar{m}]$ under the isomorphism from  Proposition \ref{H_0.prop}. Here $\bar{x}$ denotes the inverse of the loop $x.$
A priori, the isomorphism depends on what homotopy one uses in the proof of 
Proposition \ref{H_0.prop} to represent elements in $\Strata_0$ as loops based at $q \in K.$
However, since $[e]$ and $[\bar{m}]$ are in the center of the image of the peripheral group,
$\mathfrak{e}$ and $\bar{\mathfrak{m}}$ are independent of these choices.

The cobar construction $\Omega \mathcal{C}$ has no identity element.
We will formally add the identity element, although this is a ``benign"\footnote{It is a formal operation that can be done on a ring and commutes with all other constructions (on the ring) that we do here.} operation in that it can be performed at essentially any point in this note.

\begin{dfn}
Given any ring $R,$ with or without an identity element, let $R(1)$ be the ring
enlarged to include the identity element $1.$ 
That is, $R(1) = \Z \oplus R.$ 
Here the product $*$ on $R(1)$ is the usual
$(a, r) * (b, s) = (ab, as+br+rs)$
for  $r,s \in R$ and integers $a,b.$ 
\end{dfn}

\begin{dfn}
\label{StringAlgebra.dfn}
Define the {\bf{string algebra}} to be the degree zero homology of the cobar construction, $H_0(\Omega \mathcal{C})$
with identity element, modulo placing the elements $\mathfrak{e}$ 
and $\bar{\mathfrak{m}}$ in the center
\[
A_{St} = \frac{H_0(\Omega \mathcal{C})(1)}{[\mathfrak{e}, \cdot], [\bar{\mathfrak{m}}, \cdot]}.
\]
\end{dfn}

\begin{figure}[h]
\begin{center}
\includegraphics[scale=0.5]{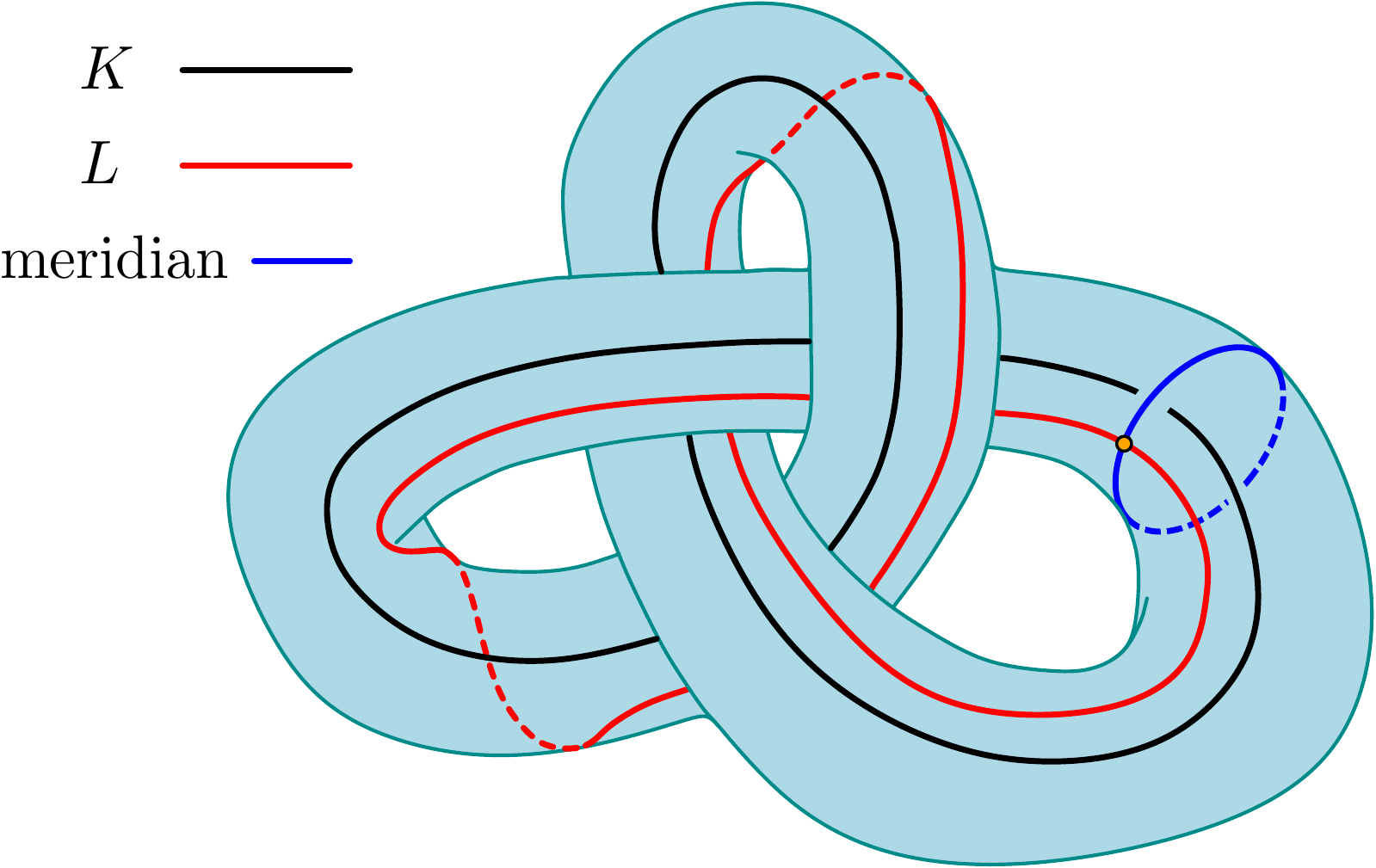}\vspace{0.2cm}
\caption{The case of a framed knot $K$ in $\R^3$.}
\end{center}\vspace*{-0.3cm}
\end{figure}

\begin{prop}
\label{Ast.prop}
$A_{St}$ is isomorphic to 
$\widetilde{\Z \pi_1}(M \setminus K)(1),$  twisted by the element $[e] + [m],$ modulo the two-sided ideals generated by
\begin{eqnarray*}
&\left\{[lx]-[x], [xl] - [x]\,\, \left| \,\, [x] \in \pi_1(M \setminus K),  [l] \in L\right\}, \right. & 
\mbox{and}\\
&\left\{[mx]-[xm]\,\, \left| \,\, [x] \in \pi_1(M \setminus K)\right\} \right. . &
\end{eqnarray*}
\end{prop}

This result follows from Proposition \ref{H_0.prop} and Lemma \ref{eCommute.lma}  below.

\begin{lma}
\label{eCommute.lma}
Consider the group ring $\widetilde{\Z \pi_1}(M \setminus K)(1)$
twisted by the element $[e] + [m].$ 
The two-sided ideals generated by 
\[
\left\{[mx]-[xm]\,\, \left| \,\, [x] \in \pi_1(M \setminus K)\right\} \right. \quad \mbox{and}
\]
\[
\left\{[x] \tilde{*} [e] -[e] \tilde{*} [x] \,\, \left| \,\, [x] \in \pi_1(M \setminus K) \right\}. \right.
\] 
are the same; similarly the two-sided ideals generated by 
\[
\left\{[\bar{m}x]-[x\bar{m}]\,\, \left| \,\, [x] \in \pi_1(M \setminus K)\right\} \right. \quad \mbox{and}
\]
\[
\left\{[x] \tilde{*} [\bar{m}] -[\bar{m}] \tilde{*} [x] \,\, \left| \,\, [x] \in \pi_1(M \setminus K)\right\}. \right. 
\] 
are the same.
\end{lma}

\begin{proof}
Note that 
\[
[e]\tilde{*}[x] - [x]\tilde{*}[e] = [x] + [mx]  - ([x] + [xm])  = [mx] - [xm]
\]
So the generating set of the ideals match.
A similar computation justifies the second statement.
\end{proof}

Next we show how $A_{St}$ is in fact an algebra over a larger ring of coefficients.

\begin{prop}
\label{AstCoefficients.prop}
The algebra $A_{St}$ is an algebra over 
$\Z[\mu^{\pm 1}]$, the group ring of an infinite cyclic group.
\end{prop}

\begin{proof}

Define the map $f: \Z[\mu^{\pm 1}] \rightarrow A_{St}$
on the generators
\[
f(\mu) = \mathfrak{e} - 1, \quad f(\mu^{-1}) = \bar{\mathfrak{m}} - 1
\]
and extend it as a ring homomorphism.
The map is well-defined because it preserves the one relation
\[
f(\mu) f (\mu^{-1}) = (\mathfrak{e} - 1) \tilde{*}(\bar{\mathfrak{m}} -1) = [ee\bar{m}] + [em\bar{m}] - [e] - [\bar{m}] + 1 = 1,
\]
where here we use the isomorphism from Proposition \ref{H_0.prop} to identify $[e]$ with  $\mathfrak{e}$ and $[\bar{m}]$ with $\bar{\mathfrak{m}}.$
Since $\mathfrak{e}$ and  $\bar{\mathfrak{m}}$ lie in the center of $A_{St}$ we are done.
\end{proof}

\subsection{An Example}
\label{Example.sec}

We end this section with an example.
\begin{prop}
\label{Unknot.prop}
The string algebra $A_{St}$ of the unknot $K$ in $M=\R^3$
is the group ring of an infinite cyclic group.
\end{prop}

\begin{proof}

Note that the ideal that we divide out by in the statement of Proposition \ref{Ast.prop} are trivial in the case of the unknot: $L$ is the empty set and $\pi_1(M\setminus K)$ is commutative so $[mx] - [xm] = 0.$
Thus $A_{St}$ is $\widetilde{\Z \pi_1}(M \setminus K)(1)$
twisted by $[e] + [m].$ 

Consider the following elements which generate $A_{St}$ additively: 
\begin{eqnarray*}
& m\{1\} = [e],  m\{2\} = [m], m\{3\} = [m^2], m\{4\} = [m^3], \ldots, & \\
&\bar{m}\{1\} = [\bar{m}], \bar{m}\{2\} = [\bar{m}^2], \bar{m}\{3\} = [\bar{m}^3], \ldots, &\\
&m\{0\} = \bar{m}\{0\} = 1.&
\end{eqnarray*}
Note that these elements are linearly independent in the underlying vector space of $\Z \pi_1(M \setminus K)(1),$
and hence in the (same) space which underlies $\widetilde{\Z \pi_1}(M \setminus K)(1).$


For a fixed $k \ge 1,$ define $\Z[\mu^{\pm 1}]\{k\}$ to be the linear vector space spanned by $1, \mu^1, \ldots, \mu^k, \mu^{-1}, \ldots, \mu^{-k}$. Similarly, define $A_{St}\{k\}$ to be the linear vector space spanned by $1, m\{1\}, \ldots, m\{k\}, \bar{m}\{1\}, \ldots, \bar{m}\{k\}$.

One can compute that
\[
m\{k+1\} = m\{k\} \tilde{*} (m\{1\} - 1), \quad
\bar{m}\{k+1\} = \bar{m}\{k\}\tilde{*} (\bar{m}\{1\}  - 1).
\]
from which it follows
\begin{equation}
\notag
f(\mu^k) = \sum_{i=0}^k (-1)^{k-i} m\{i\}, \quad
f(\mu^{-k}) = \sum_{i=0}^k (-1)^{k-i} \bar{m}\{i\}
\end{equation}
for the map $f: \Z[\mu^{\pm}] \rightarrow A_{St}$ defined in Proposition \ref{AstCoefficients.prop}.
Thus, we get the restriction  $f|_{ \Z[\mu^{\pm 1}]\{k\}}$ has range $A_{St}\{k\}.$
Since $f(\mu^{k-1}(\mu - 1)) = m\{k\}$ and 
$f(\mu^{-(k-1)}(\mu^{-1} - 1)) = \bar{m}\{k\},$ this restriction function is
onto. Thus, the restriction is a linear isomorphism of $(2k+1)$-dimensional vector spaces.

Since $\Z[\mu^{\pm 1}]\{1\} \subset \Z[\mu^{\pm 1}]\{2\} \subset \ldots \subset \Z[\mu^{\pm 1}]$ and
$A_{St}\{1\} \subset A_{St}\{2\} \subset \ldots \subset A_{St}$ are exhaustive
inclusions of the underlying vector spaces, the isomorphisms for each $k$ prove that $f$ is a ring isomorphism. 

\end{proof}

\section{Comparing the string algebra and the Ng cord algebra}
\label{Ng.sec}
In this section we review 
 Ng's fundamental group presentations of the cord algebra and prove the
 Main Theorem \ref{MainTheorem.thm}, which is to equate it with the string algebra $A_{St}$ from Section \ref{StringAlgebra.sec}.
All manifolds are assumed to be connected.
 
\subsection{Fundamental group formulation of the Ng cord algebra}
\label{FungamentalGroupFormulation.sec}


For a knot in $\R^3$, Ng constructed a combinatorial DGA over $\Z[\lambda^\pm, \mu^\pm, Q^\pm]$ whose homology is equivalent to knot contact homology, a version of relative symplectic field theory which is defined using pseudo-holomorphic curves in symplectic manifolds \cite{EENS}.
In \cite{Ng2} he interprets the degree zero homology with $\Z$-coefficients in terms of an algebra generated by cords divided by an ideal generated by skein relations.
Ng and Gadgil in \cite[Appendix]{Ng2} describe a fundamental group presentation of the cord algebra, which may be applied to any codimension 2 submanifold $K$ of $M$.
In \cite[Definition 2.2]{Ng3} Ng extends this presentation to define the cord algebra over $\Z[H_1(\partial N)]$ where $N$ is a tubular neighborhood of $K$.
For a knot in $\R^3$ this presentation gives the $Q=1$ specialization of the degree zero part of knot contact homology.

We review below the construction of the cord algebra for a codimension 2 submanifold $K$ of $M$.
Instead of working with $\Z[H_1(\partial N)]$-coefficients, we work in the specialization to $\Z[\mu^\pm]$-coefficients, where $\mu$ is the homology class of a fiber of $\partial N \to K$.
Note that in the knot case, this theory is still more general than the cord algebra with $\Z$-coefficients, which is sufficient to distinguish knots indistinguishable by those invariants listed in Corollary \ref{NgComputation.cor}.

\begin{dfn}
\label{CordAlgebra.dfn}
Let $V_\mu$ be the free module generated by $\pi_1(M \setminus K)$ over the group ring $\Z[\mu^{\pm 1}].$
Let $T(V_\mu)$ denote the unital tensor algebra of this module. 
Let $[x],[y]$ be arbitrary elements in $\pi_1(M \setminus K).$
Define the (Ng) {\bf cord algebra} $A_{Ng}( K)$
of the submanifold as $T(V_\mu)$ modulo the relations
\begin{eqnarray}
\label{key.eq}
&[x y] + [x m y] =
[x] \otimes [y]& \\
\label{longitude.eq}
&[lx] = [xl] = x&\\
\label{meridian1.eq}
&[mx] = [xm], \,\, [\bar{m}x] = [x\bar{m}] &\\
\label{weird.eq}
&[e] =1 + \mu& \\
\label{meridian2.eq}
&[mx]  =\mu [x], \,\,[\bar{m}x] = \mu^{-1}[x]&
\end{eqnarray}
Here $e$, $m$ and $l$ are as defined in Section \ref{Cobar.sec}.
\end{dfn}
We note that not all the relations in Definition \ref{CordAlgebra.dfn} are independent. 
For example, relation (\ref{meridian2.eq}) follows from relations (\ref{key.eq}) and (\ref{weird.eq}).
Such a  presentation, however, will be useful when comparing the cord algebra to string topology constructions.

\subsection{Main Theorem}
\label{MainTheorem.sec}

\begin{thm}
\label{MainTheorem.thm}
The cord algebra $A_{Ng}$ and string algebra $A_{St}$ are isomorphic algebras over the group ring of an infinite cyclic group.
\end{thm}

%

\begin{proof}
Let $V$ denote the module freely generated by $\pi_1(M\setminus K)$ over $\Z$
and $T(V)$ denote its unital tensor product.
 Proposition \ref{Ast.prop} implies that $A_{St}$ is isomorphic to $T(V)$ modulo the two-sided ideals generated by 
 (\ref{key.eq}), (\ref{longitude.eq}) and (\ref{meridian1.eq}). 
 We write this as
 \begin{equation}
 \label{ASt1.eq}
 A_{St} \cong \frac{T(V)}{(\ref{key.eq}), (\ref{longitude.eq}), (\ref{meridian1.eq}).}
 \end{equation}
As mentioned earlier, we can add the identity element before or after imposing any of these relations. (In Proposition \ref{Ast.prop} the identity element is added after dividing by the ideal associated to (\ref{key.eq}),
whereas $T(V)$ is already equipped with an identity element.)

Since the ideals associated to relations (\ref{key.eq}), (\ref{longitude.eq}) and (\ref{meridian1.eq}) are defined not using $\mu,$
\begin{equation}
\label{ASt2.eq}
\frac{T(V_\mu)}{(\ref{key.eq}), (\ref{longitude.eq}), (\ref{meridian1.eq})}
\cong
 \frac{T(V) \otimes_\Z \Z[\mu^{\pm 1}]}{(\ref{key.eq}), (\ref{longitude.eq}), (\ref{meridian1.eq})}.
\end{equation}

The result of the theorem follows from equations (\ref{ASt1.eq}) and (\ref{ASt2.eq}) if we can prove the following map is an isomorphism. Let
\begin{equation}
f :  \frac{\left( \frac{T(V) \otimes_\Z \Z[\mu^{\pm 1}]}{(\ref{key.eq}), (\ref{longitude.eq}), (\ref{meridian1.eq})} \right)}{(\ref{weird.eq}), (\ref{meridian2.eq})}
\longrightarrow
\frac{T(V)}{(\ref{key.eq}), (\ref{longitude.eq}), (\ref{meridian1.eq})} 
\end{equation}
be defined on a generating set of $V$ by
$f([x]) = [x],$ and on the coefficients by $f(\mu) = [e] - 1$ and
$f(\mu^{-1}) =[\bar{m}] - 1.$ 

To see that $f$ is well-defined, note that
\begin{eqnarray*}
f(\mu [x] - [mx]) & = & ([e] - 1) \tilde{*}[x] - [mx] = [x] + [mx] - [x] - [mx] = 0\\
f(\mu^{-1} [x] - [\bar{m}x]) & = & ([\bar{m}] - 1) \tilde{*}[x] - [\bar{m}x] =[\bar{m}x]  +[x] - [x] - [\bar{m}x] = 0\\
f([e] - (1+\mu)) & =  & [e] - (1 +([e]-1)) = 0.
\end{eqnarray*}
Since $([e] - 1) \tilde{*} ([\bar{m}] - 1) = 1 = \mu \tilde{*} \mu^{-1},$ the map is an algebra morphism. 
It is also clearly bijective.

\end{proof}






%

 
\section{Acknowledgements}

The authors would like to thank Lenny Ng and Jim Stasheff for valuable comments and remarks that helped improve the narrative of the paper.

\end{document}